\documentclass[12pt]{article}
\usepackage{amsmath, amsthm}
\usepackage[linesnumbered,boxruled]{algorithm2e}

\def\fullpage {
\addtolength{\topmargin}{-2 cm}
\addtolength{\oddsidemargin}{-0.9cm} \addtolength{\textwidth}{+2 cm}
\addtolength{\textheight}{+4 cm}}
\fullpage
\newtheorem{thm}{Theorem}
\newtheorem{lemma}[thm]{Lemma}
\newtheorem{conj}[thm]{Conjecture}
\newtheorem{prop}[thm]{Proposition}
\theoremstyle{remark}

\theoremstyle{definition}
\newtheorem{definition}[thm]{Definition}
\newcommand{\Ex}{\mathop{\bf E\/}}
\newcommand{\Var}{\mathop{\bf Var\/}}
\begin{document}

\title{Counting independent sets in triangle-free graphs}
\author{Jeff Cooper\thanks{Department of Mathematics, Statistics, and Computer
Science, University of Illinois at Chicago, IL 60607;  email:
jcoope8@uic.edu},
Dhruv
Mubayi\thanks{Department of Mathematics, Statistics, and Computer
Science, University of Illinois at Chicago, IL 60607;  email:
mubayi@math.uic.edu; research  supported in part by  NSF grant DMS
0969092.}}

\maketitle

\vspace{-0.4in}

\begin{abstract}

Ajtai, Koml\'os, and Szemer\'edi proved that for sufficiently large $t$ every triangle-free graph with $n$ vertices and average degree $t$ has an independent set of size at least $\frac{n}{100t}\log{t}$.
We extend this by proving that the number of independent sets in such a graph is at least
\[
2^{\frac{1}{2400}\frac{n}{t}\log^2{t}}.
\]
This result is sharp for infinitely many $t,n$ apart from the constant.
An easy consequence of our result is that there exists $c'>0$ such that every $n$-vertex triangle-free graph has at least
\[
2^{c'\sqrt n \log n}
\]
independent sets. We conjecture that the exponent above can be improved to $\sqrt{n}(\log{n})^{3/2}$. This would be sharp by the celebrated result of Kim which shows that the Ramsey number $R(3,k)$ has order of magnitude $k^2/\log k$.

\end{abstract}

\section{Introduction}
An independent set in a graph $G=(V,E)$ is a set $I \subset V$ of vertices such that no two vertices in $I$ are adjacent. The independence
number of $G$, denoted $\alpha(G)$, is the size of the
largest independent set in $G$.  Determining the independence number of a graph is one of the most pervasive and fundamental problems in graph theory. The independence number  naturally arises when studying other fundamental graph parameters like the chromatic number (minimum size of a partition of $V$ into independent sets), clique number (independence number of the complementary graph), minimum vertex cover  (complement of a maximum independent set), matching number (independence number in the line graph) and many others.

Throughout this paper, we suppose that $G$ is a graph with $n$ vertices and average degree $t$. Turan's \cite{turan} basic theorem of extremal graph theory, in complementary form, states that $\alpha(G) \geq \lceil n/(t+1)\rceil $
for any graph $G$. This bound is tight, as demonstrated by the complement of Turan's graph $G=\overline{T(n, r)}$ which, in the case $n=kr$ is the disjoint union of $r$ cliques, each with $k$ vertices (then $\alpha(G)=r$ and $t=k-1$).
 Since $G$ contains large cliques it is natural to ask whether Tur\'an's bound on $\alpha(G)$ can be improved if we prohibit cliques of a prescribed (small) size in $G$.

In \cite{trifreeaks}, Ajtai, Koml\'os, and Szemer\'edi showed that if $G$ contains no $K_3$, then this is indeed the case, by improving Tur\'an's bound by a factor that is logarithmic in $t$. More precisely, they proved that if $G$ is triangle-free, then
$$\alpha(G) \geq \frac{n}{100t}\log{t}.$$
 Shortly after, Shearer \cite{trifreeshearer} improved this to
$\alpha(G) \geq (1-o(1))\frac{n}{t}\log{t}$ (assume for convenience throughout this paper that  $\log=\log_2$). Random graphs \cite{trifreeshearer} show that for infinitely many $t$ and $n$ with $t=t(n)\rightarrow \infty$ as $n \rightarrow \infty$, there are $n$-vertex triangle-free graphs
with average degree $t$ and independence number $(2-o(1))((n/t)\log t)$.  Consequently, the results of \cite{trifreeaks, trifreeshearer}  cannot be improved apart from the multiplicative constant.

There is a tight connection between the problem of determining $\alpha(G)$  and questions in Ramsey theory. More precisely, determining the minimum possible $\alpha(G)$ for a triangle-free $G$ is equivalent to determining the Ramsey number $R(3, k)$, which is the minimum $n$ so that
every graph on $n$ vertices contains a triangle or an independent set of size $k$. Moreover,  the above lower bounds for $\alpha(G)$ are equivalent to the upper bound $R(3,t) = O(t^2/\log t)$. It was a major open problem, dating back to the 1940's, to determine the order of magnitude of $R(3,t)$, and this was achieved by
 Kim \cite{trifreekim} who showed that for every $n$
sufficiently large, there exists an $n$-vertex  triangle-free graph $G$ with  $\alpha(G)<9\sqrt{n \log n}$.  As a consequence, the upper bound  $R(3,t)=O(t^2/\log t)$ from \cite{trifreeaks} is of the correct order of magnitude.

In this paper, our goal is to take the result of Ajtai, Koml\'os, and Szemer\'edi~\cite{trifreeaks} further by not only finding an independent set of the size guaranteed by their result, but by showing that many of the vertex subsets of approximately that size are independent sets.
\begin{definition}
Given a graph $G$, let $i(G)$ denote the number of independent sets in $G$.
\end{definition}

Upper bounds for $i(G)$ have been motivated by combinatorial group theory.
In \cite{upperalon}, Alon showed that if $G$ is a  $d$-regular graph, then $i(G) \leq 2^{(1/2+o(d))n}$;
he also conjectured that
$$i(G) \leq (2^{d+1}-1)^{n/2d}.$$
Kahn~\cite{upperkahn} proved this conjecture for $d$-regular bipartite graphs.
Galvin~\cite{uppergalvin} obtained a similar bound for any  $d$-regular graph $G$, namely
$$i(G) \leq 2^{n/2(1+1/d+c/d\sqrt{\log{d}/d})}$$ for some constant $c$ .
Finally, Zhao~\cite{upperzhao} recently resolved Alon's conjecture.

In this paper, we consider lower bounds for $i(G)$.
This problem is fundamental in extremal graph theory, indeed, the Erd\H os-Stone theorem ~\cite{linearerdos} gives a lower bound for $i(G)$ that is the correct order of magnitude {\em provided $n/t$ is a constant}. More recently, the problem in the range $t=\Theta(n)$ has been investigated by Razborov~\cite{tridensityrazborov},
Nikiforov~\cite{tridensitynikiforov}, and Reiher.
For example, the results of Razborov and Nikiforov determine $g(\rho,3)$,
the minimum triangle density of an $n$-vertex graph with edge density $\frac{1}{2}<\rho<1$.
Looking at the complementary graph, this gives tight lower bounds on the number of independent
sets of size three in a graph with density $1-\rho\sim\frac{t}{n}$.

Lower bounds
 for $i(G)$ appear not to have been studied with the same intensity when $t$ is much smaller than $n$, in particular, when $t \rightarrow \infty$ and $t/n \rightarrow 0$.
Let us make some easy observations that are relevant for our work here. We assume that $\alpha:=\alpha(G) \le n/4$.
Since every subset of an independent set is also independent, Turan's theorem implies
\[
i(G) \geq 2^{\alpha} \geq 2^{n/(t+1)}.
\]
In Section \ref{general}, we will improve this to
\begin{equation} \label{easyprop}
i(G) \geq 2^{\frac{1}{250}\frac{n}{t}\log{t}}.
\end{equation}
Our proof uses the standard probabilistic argument which establishes the order of magnitude given by Tur\'an's bound on $\alpha(G)$. This result is certainly not new, and we present it only to serve as a warm-up for our main result in Section \ref{trifree}.  Let us observe below that the result is essentially tight.

As no subset of size more than $\alpha(G)$ is independent, an easy upper bound on $i(G)$ (using $\alpha \le n/4$)
is
\begin{equation}\label{easyup}
i(G) \leq \sum_{i=0}^{\alpha}{\binom{n}{i}} \leq
	2\binom{n}{\alpha}.
\end{equation}
Since $\alpha(\overline{T(kr,r)}) = r = n/(t+1)$ (recall that $n=kr$ and $t=k-1$), this bound implies that as $n \rightarrow \infty$
\[
i(\overline{T(kr,r)}) \leq 2\binom{kr}{r} \leq 2(ek)^{r} = 2^{1+r\log{ek}} = 2^{(1+o(1))\frac{n}{t}\log{t}}.
\]
Thus, apart from the constant, the exponent in \eqref{easyprop} cannot be improved.

Our main result addresses the case where $G$ contains no triangles. As in the case
of the independence number, prohibiting triangles improves the bound in \eqref{easyprop}.
\begin{thm} {\bf (Main Result)} \label{mr}
Suppose that $G$ is a triangle-free graph on $n$ vertices with average degree
$t$, where $t$ is sufficiently large. Then
\begin{equation} \label{formula}i(G) \geq 2^{\frac{n}{2400t}\log^2{t}}.\end{equation}
\end{thm}
\noindent
Suitable modifications of Random graphs provide constructions of $n$-vertex triangle-free graphs $G$ with average degree $t=t(n)\rightarrow \infty$ as $n \rightarrow \infty$, and $\alpha(G)=O((n/t) \log t)$. Plugging this into (\ref{easyup}), we see that Theorem \ref{mr}
is tight (apart from the constant) for infinitely many $t$. However, it remains open if the theorem is sharp for all $t$ where $t=n^{1/2+o(1)}$.  Indeed, the open problem that remains is to obtain a sharp lower bound on $i(G)$ for triangle-free graphs with no restriction on degree. Since all subsets of the neighborhood of a vertex of maximum degree are independent, $i(G) > 2^t$.  Combining this with (\ref{formula}) we get
$$i(G)> \max\{2^t, 2^{\frac{n}{2400t}\log^2{t}} \} > 2^{c n^{1/2} \log n}$$
for some constant $c>0$. We conjecture that this can be improved as follows.

\begin{conj} There is an absolute positive constant $c$ such that every $n$-vertex triangle-free graph $G$ satisfies
$$i(G) > 2^{c n^{1/2} (\log n)^{3/2}}.$$\end{conj}

The conjecture, if true, is sharp (apart from the constant in the exponent) by the graphs (due to Kim~\cite{trifreekim} and more recently Bohman~\cite{trifreebohman})
which show that $R(3,t)=\Omega(t^2/\log t)$. Indeed, their graphs are triangle-free and have independence number
\[
\alpha(G) = \Theta(t) = \Theta(\frac{n}{t}\log{t}) = \Theta(\sqrt{n\log{n}}),
\]
so $i(G) \leq 2^{O(\sqrt{n}\log^{3/2}{n})}$ by \eqref{easyup}.

As mentioned before, throughout the paper, all logarithms are base 2. For a graph $G$,
let $n(G), e(G)$ and $t(G)$ denote the number of vertices, edges, and average degree of $G$.

\section{General case}\label{general}
In this section, we give the simple proof of (\ref{easyprop}).  Our purpose in doing this is to familiarize the reader with
the general approach to the proof of Theorem \ref{mr} in the next section.

\begin{prop}
If G is a graph on n vertices with average degree $t$, where $2 \leq t \leq \frac{n}{800}$, then
$i(G) \geq 2^{\frac{1}{250}\frac{n}{t}\log{t}}$.
\end{prop}
\begin{proof}
	Set $k=\lfloor \frac{1}{100}\frac{n}{t} \rfloor$. Pick a $k$-set uniformly at random from all $k$-sets
	in $V(G)$. Let $H$ be the subgraph induced by the $k$ vertices.
	Then
	\begin{equation*}
	\Ex[e(H)] = \frac{1}{2}nt\frac{\binom{n-2}{k-2}}{\binom{n}{k}}
	= \frac{1}{2}nt\frac{k(k-1)}{n(n-1)} <\frac{1}{2}\frac{tk^2}{n}.
	\end{equation*}
Recall that Markov's inequality states that if $X$ is a positive random variable and $a>0$,
then $\Pr[X \geq a] \leq \Ex[X]/a$; hence $\Pr[e(H) \geq 2\frac{1}{2}\frac{tk^2}{n}] \leq 1/2$.
So for at least half of the choices for $H$, $e(H) \leq \frac{tk^2}{n}$. Therefore,
the number of choices of $H$ for which $e(H) \leq \frac{tk^2}{n}$ is at least
\begin{equation}\label{hct}
	\frac{1}{2}\binom{n}{k} \geq \frac{1}{2}(\frac{n}{k})^k >
	2^{\frac{k}{2}\log{n/k}} =
	2^{\frac{k}{2}(\log{n}-\log{k})}.
\end{equation}

Now, if $e(H) < \frac{tk^2}{n} = \frac{1}{100}k$, then at most $\frac{1}{50}k$ of the
vertices in $H$ have degree at least one. This in turn implies that $H$ contains an independent
set $I$ of size at least $\frac{49}{50}k$. The set $I$ can be obtained from any $H$
which contains it; the number of ways to pick the $\frac{1}{50}k$ vertices of $H-I$ is at most
\begin{equation}
	\binom{n}{k/50} \leq (\frac{50ne}{k})^{k/50} =
	2^{\frac{k}{50}\log{50ne}-\frac{k}{50}\log{k}} \leq
	2^{\frac{k}{50}(\log{n}-\log{k}) + \frac{k}{50}\log{100t}}.
\end{equation}
Combining this with \eqref{hct} and using $\frac{20}{23}\frac{1}{100}\frac{n}{t} < k  \leq \frac{1}{100}\frac{n}{t} $ for $t \leq \frac{n}{800}$,
\begin{equation*}
	i(G) \geq
	 2^{k(\frac{1}{2}-\frac{1}{50})(\log{n}-\log{k})-\frac{k}{50}\log{100t}} \geq
	2^{k\frac{24}{50}\log{100t}-\frac{k}{50}\log{100t}} =
	2^{k\frac{23}{50}\log{100t}} >
	2^{\frac{1}{250}\frac{n}{t}\log{t}}.
\end{equation*}
\end{proof}

\section{Triangle-free graphs}\label{trifree}
In this section we prove our main result, Theorem \ref{mr}.  We begin with some modifications of a
lemma from \cite{trifreeaks} (see the proof of Lemma 4 in \cite{trifreeaks}).
\begin{lemma}\label{akslemma} {\bf (Ajtai-Koml\'os-Szemer\'edi~\cite{trifreeaks})}
Suppose that $G$ is a triangle-free graph on $n$ vertices with average degree $t$, and
let $k \leq n/100t$. Let $H$ be the subgraph consisting of $k$ vertices chosen
uniformly at random from all the $k$-sets contained in $\{v \in V(G): deg(v) \leq 10t\}$.
Let $M$ be the subgraph of $G$ consisting of vertices adjacent to no vertex in $H$.
Let $n'$ and $t'$ denote the number of vertices and average degree of $M$. Then the
random variables $H$ and $M$ satisfy:
\begin{align}
	\Ex[n(M)] &> n(1-\frac{k}{n-t})^{t+1} > \frac{9n}{10} \label{exvm}  \\
	\Ex[e(M)] &> \frac{nt}{2}(1-\frac{k}{n-20t})^{20t+1} > \frac{nt}{10} \label{exem} \\
	\Ex[e(H)] &= \frac{1}{2}nt\frac{k(k-1)}{n(n-1)} \leq \frac{tk^2}{n} \label{exeh}  \\
	\Var[n(M)] &< \frac{2nk(t+1)(10t+1)}{n-k-20t-2} < nt \label{varvm} \\
	\Var[e(M)] &< 2400kt^4 < 40nt^3 \label{varem} \\
	\Var[e(H)] &\leq \frac{tk^2(10k+n)}{n^2} \label{vareh}
\end{align}
Further, if $e(M) < (1+\delta)\Ex[e(M)]$ and $n(M) > (1-\delta)\Ex[n(M)]$, then
$n'/t' > \nu n/t$, where $\delta = 800\sqrt{t/n}$ and $\nu=1-1/t-c_{10}\sqrt{t/n}$
for some positive constant $c_{10}$.
\end{lemma}
\bigskip

\noindent
{\bf Remark.}
Ajtai-Koml\'os-Szemer\'edi state their lemma for $k = n/100t$ and prove each of the
first inequalities in \eqref{exvm}, \eqref{exem}, \eqref{varvm}, and \eqref{varem} for all $k$.
They prove each of the second inequalities for $k = n/100t$, but it is easily observed that
they continue to hold for $k < n/100t$.
\bigskip

 The next lemma is implied by the computation in the proof of Lemma 4 from \cite{trifreeaks}.
However, the last statement of Lemma \ref{thelemma} is crucial to our proof of Theorem \ref{mr}, so we
make the computations in \cite{trifreeaks} explicit.
\begin{lemma}\label{thelemma}
Suppose $G$ is a triangle-free graph on $n>2^{50}$ vertices with average degree $t \leq 2\sqrt{n}\log{n}$ and
$k \leq n/100t$. Then $G$ contains a subgraph $H$ with $n(H) = k$
$e(H) \leq k/50$.
Moreover, if $M$ is the subgraph of $G$ consisting of vertices adjacent to no vertex
in $H$, then
\begin{enumerate}
	\item \label{vertcond} $n(M) > n(G)/2$ and
	\item \label{ratiocond} $n(M)/t(M) > \nu n/t$, where $\nu = 1-1/t-c_{10}\sqrt{t/n}$.
\end{enumerate}
Further, if the vertices in $H$ are chosen uniformly at random from all the $k$-sets
contained in $\{v \in V(G): deg(v) \leq 10t\}$, then at least half of the choices
for $H$ satisfy $e(H) \leq k/50$, along with conditions \ref{vertcond} and \ref{ratiocond}.
\end{lemma}
\bigskip

\begin{proof}
	Recall that for a random variable $X$ and $a>0$, Chebyshev's inequality states that
	$\Pr[|X-\Ex[X]| \geq a] \leq \Var[X]/a^2$. Thus, with $a = k/50-\Ex[e(H)]$,
	Lemma \ref{akslemma} implies
	\begin{align*}
	\Pr[e(H) \geq k/50] &\leq \frac{tk^2(10k+n)}{n^2(k/50-\frac{k^2t}{2n})^2} \\
	&=\frac{t(10k+n)}{n^2(1/50 - \frac{kt}{2n})^2} \\
	&\leq	\frac{t(n/10t +n)}{n^2(1/50-1/200)^2} \\
	&= \frac{1/10 + t}{n(3/200)^2} \\
	&< 5000\frac{t}{n} \\
	&\leq 5000\frac{2\log{n}}{\sqrt{n}} \\
	&\leq	1/1000.
	\end{align*}
	So with probability at most 1/1000, the condition $e(H) \leq k/50$ fails.
	
	Set $\delta = 800\sqrt{t/n}$. Again by Lemma \ref{akslemma} and Chebyshev,
	with $a=\delta\Ex[n(M)]$,
	\begin{equation*}
	\Pr[n(M) \leq n/2] \leq \Pr[n(M) \leq (1-\delta)\Ex[n(M)]] \leq
	\frac{nt}{(\frac{9n}{10})^2800^2\frac{t}{n}} \leq 1/1000.
	\end{equation*}
	Thus the probability that condition \eqref{vertcond} fails is at most 1/1000.
	
	With $a=\delta\Ex[e(M)]$,
	\begin{equation*}
	\Pr[e(M) \geq (1+\delta)\Ex[e(M)]] \leq
	\frac{40nt^3}{\frac{nt}{10}800^2\frac{t}{n}} = 1/160.
	\end{equation*}
	Since $\Pr[e(M) \geq (1+\delta)\Ex[e(M)] \text{ or } n(M) \leq n/2] <
	1/160 + 1/1000$, the last assertion of Lemma \ref{akslemma} implies
	that the probability of condition \eqref{ratiocond} failing is at most
	$1/160 + 1/1000$. Therefore, the probability that condition $e(H) \leq k/50$
	fails or condition \eqref{vertcond} fails or condition \eqref{ratiocond} fails
	is at most $1/1000 + 1/1000 + 1/160 + 1/1000 < 1/2$.
\end{proof}

Our proof of Theorem \ref{mr} is achieved by analyzing Algorithm \ref{algorithm} below.
The algorithm is a slight modification of the algorithm from \cite{trifreeaks} that
yields an independent set of size $\frac{1}{100}\frac{n}{t}\log{t}$. Recall that $c_{10}$
is the constant that appear in Lemma \ref{akslemma}.

\begin{algorithm} [ht]
	\KwIn{Triangle-free graph $G$ with $n$ vertices, average degree $t$}
	\KwOut{Independent set $I$}
	$M_o = G$\;
	$R = \lfloor (\log{t})/2 \rfloor$\;
	\For{$i\leftarrow 0$ \KwTo $R$} {
		$n_i = $ number of vertices in $M_i$\;
		$t_i = $ average degree in $M_i$\;
		$\nu_i = 1 - 1/t_{i-1} - c_{10}\sqrt{t_{i-1}/n_{i-1}}$\;
		\uIf{$i = 0$ \emph{or} $\nu_i > 1-1/\log{t}$} { \nllabel{condnu}
			Apply Lemma \ref{thelemma} with $G = M_i$ and $k = \lfloor \frac{1}{200}\frac{n}{t} \rfloor$\;
			$M_{i+1}, H_{i+1} = M, H$ from Lemma \ref{thelemma}\; \nllabel{getmh}
		}
		\Else{
			$I = $ Independent set in $M_{i-1}$ of size $\lceil n_{i-1}/(t_{i-1}+1) \rceil$\;
			\Return{$I$}\; \nllabel{retturan}
		}
	}
	$H = H_1 \cup \dots \cup H_R$\; \nllabel{lineh}
	$I = $ Independent set in $H$ of size $\lceil \frac{48}{50}kR \rceil$\; \nllabel{linei}
	\Return{$I$}\; \nllabel{reth}
	\caption{Independent set algorithm}
	\label{algorithm}
\end{algorithm}

\begin{thm}\label{algthm}
Suppose Algorithm \ref{algorithm} is run on a triangle-free graph $G$ with
$n$ vertices and average degree $t$, where $2^{100} < t < \sqrt{n}\log{n}$ and $n > (3c_{10})^{12}$.
If Algorithm \ref{algorithm} terminates at line \ref{retturan},
then $|I| > \frac{1}{2}\frac{n}{t}\log^2{t}$. Otherwise, for each iteration $i=0,...,R-1$,
Algorithm \ref{algorithm} successfully applies Lemma \ref{thelemma} to the graph
$M_i$ to obtain a graph $H_{i+1}$ with $k=\lfloor \frac{1}{200}\frac{n}{t} \rfloor$ vertices.
Moreover, the graph $H$ in line \ref{lineh} is the disjoint union of the $H_i$, and
the independent set $I$ in line \ref{linei} consists of $\frac{48}{50}kR$ vertices from $H$.
\end{thm}
\begin{proof}
We break our proof into two cases, depending on whether Algorithm \ref{algorithm}
terminates at line \ref{reth} or \ref{retturan}.

\noindent {\bf Line \ref{reth}}: We need to show that Lemma \ref{thelemma} can be applied
at every iteration and that the graph $H$ in line \ref{lineh} contains an
independent set of size at least
$\frac{48}{50}kR \geq \frac{1}{500}\frac{n}{t}\log{t}$.

If $i=0$, then $k \leq n/100t$, $t \leq \sqrt{n}\log{n}$,
and $n > t > 2^{50}$, so we may apply Lemma
\ref{thelemma} to obtain graphs $M_1$  and $H_1$, where $|V(H_1)|=k$.
Suppose that $i > 0$ and that Lemma \ref{thelemma} was successfully applied at all previous iterations.
Using \ref{vertcond} of Lemma \ref{thelemma}, $i<R$, and $R=\lfloor (\log{t})/2 \rfloor <(\log{n})/2$,
\begin{equation}\label{nibound}
n_i \geq n/2^i > n/2^R > \sqrt{n} > \sqrt{t} > 2^{50}.
\end{equation}
By the condition in line \ref{condnu}, $\nu_i > 1-1/\log{t}$ for each
iteration $i$. So by \ref{ratiocond} of Lemma \ref{thelemma},
$\frac{n_i}{t_i} \geq \frac{n}{t}\nu_1\nu_2\dots\nu_i > \frac{n}{t}(1-1/\log{t})^R$. Thus
\begin{equation*}
\frac{n_i}{t_i} > \frac{n}{t}(1-1/\log{t})^R >
\frac{n}{t}(1-\frac{R}{\log{t}}) \geq \frac{1}{2}\frac{n}{t}.
\end{equation*}
In particular,
\begin{equation}\label{lemmak}
	k \leq \frac{1}{200}\frac{n}{t} \leq \frac{1}{100}\frac{n_i}{t_i},
\end{equation}
	and also, $t < \sqrt{n}\log{n}$ and $n_i < n$ yield
\begin{equation}\label{lemmat}
t_i < 2n_i\frac{t}{n} \leq 2n_i\frac{\log{n}}{\sqrt{n}}
\leq 2n_i\frac{\log{n_i}}{\sqrt{n_i}} = 2\sqrt{n_i}\log{n_i}.
\end{equation}
The inequalities \eqref{nibound}, \eqref{lemmak}, and \eqref{lemmat} ensure that we may
again apply Lemma \ref{thelemma} with $M_i, M_{i+1}, H_{i+1}, $ and $k$ playing the roles
of $G, M, H, $ and $k$, respectively. Applying Lemma \ref{thelemma} $R$ times yields a collection of
sparse graphs $H_1, H_2, \dots, H_R$, each with $k$ vertices. Each
$H_i$ contains at most $\frac{2}{50}k$ vertices of degree at least one,
so each $H_i$ contains an independent set of size at least $\frac{48}{50}k$.
By definition of $M_i$, these independent sets may be combined into one
independent set of size at least $\frac{48}{50}kR$.

\noindent {\bf Line \ref{retturan}}: Suppose that the algorithm terminates at line
\ref{retturan} during iteration $i+1$. Then $n_i$, $t_i$ (and $n_{i+1}$, $t_{i+1}$)
have been defined and $\nu_{i+1} = 1-1/t_i-c_{10}\sqrt{t_i/n_i}$.
If $t_i \leq (\frac{3}{2})^{2/3}$, then $1/t_i > 1/\log^3{t}$.
Assume $t_i > (\frac{3}{2})^{2/3}$. Then
\begin{equation}\label{ti1}
 \frac{2}{3}t_i^{-1/3} > 1/t_i.
\end{equation}
By \eqref{nibound}, $n_i > \sqrt{n}$, so for $n > (3c_{10})^{12}$,
\[
n_i^3 > n^{3/2} > (3c_{10})^6n > (3c_{10})^6t_i.
\]
This implies
\begin{equation}\label{ti2}
 \frac{1}{3}t_i^{1/3} > c_{10}\sqrt{t_i/n_i}.
\end{equation}
Combining \eqref{ti1} and \eqref{ti2} yields
\begin{equation*}
1-t_i^{-1/3} < 1-1/t_i - c_{10}\sqrt{\frac{t_i}{n_i}} = \nu_{i+1} \leq 1-1/\log{t}.
\end{equation*}
Thus $1/t_i > 1/\log^3{t}$. Since $t \geq 2^{100}$ (which implies that $t > (\log^5{t}+\log^2{t})^2$)
and $i < R \leq \frac{\log{t}}{2}$,
\begin{equation*}
\frac{n_i}{t_i+1} > \frac{n}{2^i}\frac{1}{(\log^3{t}+1)} \geq
\frac{n}{\sqrt{t}}\frac{1}{(\log^3{t}+1)} = \frac{n}{t}\frac{\sqrt{t}}{(\log^3{t}+1)} >
\frac{n}{t}\log^2{t}.
\end{equation*}
Turan's theorem now implies that $M_i$ contains an independent set of size at least
$\frac{n_i}{t_i+1} > \frac{1}{2}\frac{n}{t}\log^2{t}$.
\end{proof}

\bigskip

We now complete the proof of Theorem \ref{mr} by obtaining a lower bound on the number of
outcomes given by line \ref{reth} of Algorithm \ref{algorithm}.

\noindent
{\bf  Proof of Theorem \ref{mr}.}
Recall that we are to show that if  $G$ is a triangle-free graph on $n$ vertices with
average degree $t$ sufficiently large, then $i(G) \geq 2^{\frac{n}{2400t}\log^2{t}}$.
Assume $t > \max\{(3c_{10})^{12}, 2^{100}\}$.
Then $n > t > (3c_{10})^{12}$.
Also, if $t > \sqrt{n}\log{n}$, then $G$ has
a vertex whose neighborhood contains at least
$t > \frac{n}{\sqrt{n}}\log{n} > \frac{n}{t}\log^2{n}$ vertices. Since $G$
is triangle-free, this neighborhood forms an independent set, which
contains at least $2^{\frac{n}{t}\log^2{n}} > 2^{\frac{n}{2400t}\log^2{t}}$ subsets, which are also independent.
Thus we may assume that $t \leq \sqrt{n}\log{n}$; in particular,
$G$ satisfies the hypotheses of Theorem \ref{algthm}.

If the algorithm terminates at line \ref{retturan}, then $G$ contains
an independent set of size at least $\frac{n}{2t}\log^2{t}$; so
$G$ contains at least $2^{\frac{n}{2t}\log^2{t}} > 2^{\frac{n}{2400t}{\log^2{t}}}$ independent sets,
and we are done.
Thus we may assume that Algorithm \ref{algorithm} terminates at line \ref{reth}.
Consequently, at each iteration $i$, the algorithm applies
Lemma \ref{thelemma} to pick a sparse graph with $k=\lfloor \frac{1}{2}\frac{n}{100t} \rfloor$ vertices.
The vertices in this graph are chosen from
\[
L_i = \{v \in V(M_i): deg(v) \leq 10t_i\}.
\]
Note that
\[
n_it_i = \sum_{v\in L_i}{deg(v)}+\sum_{v\in V(M_i)-L_i}{deg(v)} \geq \sum_{v\in V(M_i)-L_i}{deg(v)} \geq (n_i-|L_i|)10t_i.
\]
This, together with \ref{vertcond} in Lemma \ref{thelemma}, implies $|L_i| \geq \frac{9}{10}n_i > \frac{9}{10}n_{i-1}/2 \geq \frac{9}{10}n/2^i$.
At least half of the $k$-sets in $L_i$ satify the conditions of Lemma \ref{thelemma},
so the number of choices for $H_i$ is at least
\[
\frac{1}{2}\binom{|L_i|}{k} \geq \frac{1}{2}\binom{.9n/2^i}{k}.
\]
Therefore, the number of choices
for the sequence $H_1,\dots,H_R$ is at least
\begin{align}\label{thct}
\prod_{i=0}^{R-1}{\frac{1}{2}\binom{|L_i|}{k}} \geq
\frac{1}{2^R}\prod_{i=0}^{R-1}{\binom{.9n/2^{i}}{k}}
 &\geq \frac{1}{2^R}(\frac{.9n}{k})^{kR}2^{-kR^2/2} \notag \\
 &= 2^{kR\log{.9n}-kR\log{k}-kR^2/2-R} \notag \\
 &= 2^{kR(\log{n}-\log{k})-kR^2/2+kR\log{.9}-R} \notag \\
 &> 2^{kR(\log{n}-\log{k})-\frac{kR}{2}\frac{\log{t}}{2}-kR-R} \notag \\
 &> 2^{kR(\log{n}-\log{k})-\frac{kR}{4}\log{t}-2kR}.
\end{align}

Recall that Algorithm \ref{algorithm} obtains an independent set $I$ of size
$\lceil \frac{48}{50}kR \rceil$ from the graph $H=H_1 \cup \dots \cup H_R$. For a fixed
$I$, the number of graphs $H$ that yield $I$ is at most the number of
possibilities for $H-I$. This is at most $\binom{n}{|V(H-I)|}$,
which is at most
\begin{align}\label{thict}
\binom{n}{\frac{2}{50}kR} \leq (\frac{50ne}{2kR})^{\frac{2}{50}kR}
 &= 2^{\frac{2}{50}kR\log{n} + \frac{2}{50}kR\log{50e} - \frac{2}{50}kR\log{2kR} } \notag \\
 &< 2^{\frac{2}{50}kR(\log{n}-\log{k}) + \frac{2}{50}kR\log{50e}} \notag \\
 &< 2^{\frac{2}{50}kR(\log{n}-\log{k}) + kR}.
\end{align}

For a fixed $H$, the number of partitions $H_1\cup\dots\cup H_R = H$ is
at most the number of partitions of $kR$ elements into $R$ sets of size $k$,
which is less than
\begin{equation}\label{tpct}
\binom{kR}{k}^R \leq (Re)^{kR} = 2^{kR\log{Re}} < 2^{kR\log{R} +2kR} < 2^{kR\frac{1}{4}R +2kR} \leq 2^{kR\frac{1}{8}\log{t} +2kR}.
\end{equation}
Since each $H$ yields an independent set, the total number of independent
sets that can be returned at line \ref{reth} of the algorithm is at least
\begin{equation*}
\frac{\text{\# of ways to obtain $H$}}
{(\text{\# of $H$ that yield a fixed $I$})(\text{\# of partitions that yield $H$})}.
\end{equation*}
Since $\frac{n}{268t} \leq k \leq \frac{n}{200t}$ and $R > (\log{t})/3$,
\eqref{thct}, \eqref{thict}, and \eqref{tpct} imply that this is at least
\begin{align*}
	2^{\frac{48}{50}kR(\log{n}-\log{k}) - \frac{5}{8}kR\log{t}-5kR}
	& \geq 2^{\frac{48}{50}kR\log{200t} - \frac{5}{8}kR\log{t}-5kR} \\
	&> 2^{\frac{134}{400}kR\log{t}} \\
	&> 2^{\frac{134}{1200}k\log^2{t}}\\
	&> 2^{\frac{1}{2400}\frac{n}{t}\log^2{t}}.	
\end{align*}
\qed
\bibliographystyle{amsplain}
\bibliography{bib}
\end{document}